\documentclass[a4paper,11pt]{amsart}

\usepackage{amsfonts}
\usepackage{amsmath}
\usepackage{amssymb}
\usepackage{amsthm}
\usepackage{amscd}
\usepackage{mathrsfs}
\usepackage[all]{xy}
\usepackage{xcolor}
\usepackage{todonotes}
\usepackage{enumitem}

\renewcommand{\P}{\mathbb{P}}
\newcommand{\G}{\mathbb{G}}
\newcommand{\Z}{\mathbb{Z}}

\newcommand{\sO}{\mathscr{O}}
\newcommand{\st}{\text{ s.t. }}
\newcommand{\aut}{\operatorname{Aut}}

\newcommand{\cN}{\mathcal{N}}
\newcommand{\cG}{\mathcal{G}}

\newcommand{\cF}{\mathcal{F}}
\newcommand{\cB}{\mathcal{B}}
\newcommand{\cP}{\mathcal{P}}
\newcommand{\cH}{\mathcal{H}}
\newcommand{\cT}{\mathcal{T}}
\newcommand{\X}{\mathcal{X}}

\newcommand{\sT}{\mathscr{T}}

\newcommand{\Wu}{\mathcal{W}_{\operatorname{U}}}
\newcommand{\Wn}{\mathcal{W}_{\operatorname{N}}}

\newtheoremstyle{dotlessP}{}{}{\color{blue}}{}{\color{blue}\bfseries}{}{ }{}

\theoremstyle{plain}
\newtheorem{teor}{Theorem}[section]

\newtheorem{prop}[teor]{Proposition}
\newtheorem{lemma}[teor]{Lemma}
\newtheorem{oss}[teor]{Remark}
\newtheorem{defn}[teor]{Definition}

\title{Non uniform projections of surfaces in $\P^3$}
\author{Alice Cuzzucoli, Riccardo Moschetti and Maiko Serizawa}

\address{A. Cuzzucoli\\ Department of Mathematics, University of Warwick, Coventry, CV4 7AL, Warwickshire, England}
	\email{a.cuzzucoli@warwick.ac.uk}

\address{R. Moschetti\\Department of Mathematics and Natural Sciences, University of Stavanger \\NO-4036 Stavanger, Norway}
\email{riccardo.moschetti@uis.no}

\address{M. Serizawa \\ School of Mathematics and Statistics, University of Sheffield\\
Western Bank, Sheffield, S10 2TN, South Yorkshire, England}
	\email{mserizawa1@sheffield.ac.uk}

\thanks{}
\begin{document}
\begin{abstract}
Consider the projection of a smooth irreducible surface in $\P^3$ from a point.
The uniform position principle implies that the monodromy group of such a projection from a general point in $\P^3$ is the whole symmetric group. We will call such points uniform. Inspired by a result of Pirola and Schlesinger for the case of curves, we proved that the locus of non-uniform points of $\P^3$ is at most finite.
\end{abstract}
\maketitle
\section{Introduction}

In this paper we study the monodromy groups of projections of smooth irreducible surfaces in $\P^3$ from points. Our particular interest lies in finding whether the monodromy group of a projection $\pi_L$ of a surface from a point $L$ is the whole symmetric group or not. In the former case, we call $L$ \textit{uniform}, and in the latter case, \textit{not uniform}. The very same terminology is also used for the whole projection, calling $\pi_L$ uniform if and only if the monodromy group $M(\pi_L)$ is the whole symmetric group.
Much work has been done in the literature for the case of curves to determine which groups can arise as the monodromy group of some projection. The papers \cite{GM_general}, \cite{GuMa} and \cite{GuSh} show that if $C$ is a general curve of genus greater than $3$, then the monodromy group of a projection $C \to \P^1$ is either the whole symmetric group or the alternating group. A general curve of degree $d$ and genus $g$ admits a covering with symmetric monodromy group if $d \geq g/2+1$. Moreover, this holds for every curve if $d \geq g+2$. The existence of a covering with alternating monodromy group has been covered in \cite{MagVol} for general curves with $d \geq 2g+1$ and in \cite{ArPi} for curves with $d \geq 12g+4$.

The monodromy groups arising from such projections have been studied in different contexts; for instance, when the subvariety is a curve all the possible monodromy groups were classified by Miura and Yoshihara in \cite{Mi1}, \cite{Mi2}, \cite{Mi3}.

In a more general context, given a subvariety $X$ of arbitrary dimension in $\P^r$, one can study the locus $\Wu(X)$ of all possible linear subspaces of $\P^r$ for which the monodromy group of the projection is the full symmetric group. We say that such linear subspaces of $\P^r$ are \textit{uniform}. Similarly, we will denote by $\Wn(X)$ the locus of \textit{non-uniform} linear subspace of $\P^r$. Indeed, each subvariety $X$ of $\P^r$ carries its own locus $\Wu(X)$ of uniform subspaces in the Grassmannian $\G(r-n-1, \P^r)$, where $n$ denotes the dimension of $X$.
Consider the case where $X$ is an irreducible algebraic curve in $\P^r$. The so-called Uniform Position Principle implies that a general $(r-2)$-plane of $\P^r$ is
uniform. A classical reference of this result is the work of Harris \cite{HarrisGenus}, where the Uniform Position Principle is proved in Section $2$. Pirola and Schlesinger strengthen this statement in \cite{PS} by proving that, in the above setting, the locus of non-uniform $(r-2)$-planes has codimension at least two in the Grassmannian $\G(r-2, \P^r)$. 

Following the same stream of ideas, it would be interesting to study the locus of non-uniform subspaces $\Wn(X)$ in the appropriate Grassmannian for a higher dimensional subvariety $X$. We focus on the case of a projection of a smooth surface $X$ in $\P^3$, where the locus $\Wn(X)$ is a subset of $\P^3$. Our main result is the following:

\begin{teor} \label{thm:main}
For every smooth irreducible surface $X$ in $\P^3$, the locus $\Wn(X)$ of non-uniform points is finite.
\end{teor}

There are many similarities between the case of surfaces in $\P^3$ and the case of plane curves. For the latter case, the aforementioned results of Pirola and Schlesinger imply that the dimension of the locus of non-uniform points is at most zero. In fact, our present work for surfaces in $\P^3$ is a first step toward the generalization of the above statement to higher dimensional hypersurfaces. We expect the dimension of the non-uniform locus to be at most zero for any smooth hypersurface in $\P^r$.

Despite the similarities described above, our proof of Theorem \ref{thm:main} is built upon techniques which are quite different from the ones used in \cite{PS}. The main tools we use in our work come from classical differential projective geometry: the study the family of multi-tangent lines to the surface $X$ and the so-called \textit{focal loci}, as developed in \cite{CF}. A standard way in classical algebraic geometry to prove that a set $V$ is finite consists of showing first that $V$ is contained in a certain algebraic set, and then that such a set is of dimension zero. In this framework, our strategy can be summarised in the following points: 

\begin{enumerate}
\item Describe a suitable filling family of lines related to the monodromy group of the surface $X$.
\item Prove that the dimension of the focal locus of this family is actually zero.
\item Prove that the locus of non-uniform points is contained in the focal locus of such a family.
\end{enumerate}

All of the above points are carried out in Sections \ref{sec:prel} and \ref{sec:nupos}.

Another viewpoint for the generalization of the uniform position principle is carried out by Cukierman in \cite{C}, in which he shows that the locus of non-uniform points for a general planar curve is in fact empty. A similar result is expected for a general hypersurface in $\P^r$, but this problem is still open. However, we prove that this is true for the case case of cubic surfaces thanks to the classification of automorphisms carried out in \cite{H}. This is done in Proposition \ref{prop:genCucCub}.
\vspace{3 mm}


\section{Preliminaries} \label{sec:prel}
\subsection{Coverings and monodromy group}
Let $X$ and $Y$ be two complex irreducible algebraic varieties of the same dimension together with a generically finite dominant map $f:X \to Y$ of degree given by the corresponding field extension $d=[K(X):K(Y)]$. Consider a point $y$ in an open dense subset $U \subset Y$ such that $f|_U$ is an \'etale covering. Then for any element $\gamma$ in the fundamental group $\pi_1(U,y)$ there exists a unique lift, coming from $f|_U$, such that if $x\in f^{-1}(y)$ is a point in the fibre of $y$ then the lifting $\tilde{\gamma}: [0,1] \to X$ satisfies $\tilde{\gamma}(0)=x$. Under these conditions we have the following well defined map
\begin{equation} \label{eqn:rho}
\begin{aligned}
\rho : \pi_1(U,y)\ &\to \aut(f^{-1}(y))\\
\gamma &\mapsto (x \mapsto \tilde{\gamma}(1)).
\end{aligned}
\end{equation}

\begin{defn}
The \textbf{monodromy group} of $f$ is the image under $\rho$ of the fundamental group $\pi_1(U,x)$, and is denoted by $M(f)$.
\end{defn}

Via the identification of $\aut(f^{-1}(y))$ with the symmetric group $S_d$ in $d$ elements, the group $M(f)$ can be viewed as a subgroup of $S_d$. It is a well-established fact that the monodromy group may be identified with the Galois group of the corresponding field extension, as proved in Section I of \cite{H_enum}. This allows to translate problems about monodromy into the study of the corresponding Galois groups: for instance, one can check from this
correspondence that the monodromy group is independent from the base point $x$ and does not depend on the open $U$. 

In order to study the locus of non-uniform points, we will consider the following sufficient condition for a group to be the whole symmetric group. This is an application of Jordan's theorem, see \cite{FPG}, Theorem 13.9. We state the result in the context of monodromy groups, and give a direct proof.

\begin{lemma} \label{lemma:IndTrans}
Let $M(f)$ be the monodromy group of a certain morphism $f$ of degree $d$. Assume $M(f)$ is generated by transpositions, then $M(f)$ is isomorphic to the symmetric group $S_d$.
\end{lemma}
\begin{proof}
We know that $M(f)$ is a transitive subgroup of $S_d$. Assuming that it is generated by transpositions, let us prove that it is all $S_d$. 

Let us consider the set $K$ of the transpositions in $M(f)$ containing the element $1$. By hypothesis we know that this set is not empty, so we can assume $(1,2) \in K$.
The aim is now to prove that $K$ contains $\{(1,2), (1,3), \ldots, (1,d)\}$.
Assume by contradiction that this is not true, so up to changing the name of the elements we have $K=\{(1,i), 2 \leq r < d\}.$
By transitivity, and by the fact that $M(f)$ is generated by transpositions, there is an $(a,b)$ with $a \leq r$ and $b>r$. Then $M(f)$ contains also the transposition $(1,a)(a,b)(1,a)=(1,b)$, which was not in the set. This is a contradiction.
\end{proof}

In this paper, we work with projections from a point $L$ in $\P^3$. This is a rational map $\pi_{L}:\P^3 \dashrightarrow \P^2_L$, where $\P^2_L$ is the projective space of lines on $\P^3$ containing $L$. Taking a point $x \in \P^3 \smallsetminus L$ one can consider the line $\langle x, L \rangle$ spanned by $x$ and $L$. 

\begin{lemma} \label{lemma:projfinite}
Let $X \subset \P^3$ be a smooth, irreducible surface of degree $d$ and consider a point $L \in \P^3 \smallsetminus X$. The restriction of $\pi_{L}$ to $X$ is a dominant surjective morphism of finite degree $d$.
\end{lemma}
\begin{proof}
Since $L \notin X$, the projection is well defined. Consider a point $t$ in $\P^{2}_L$, corresponding to a line $\lambda \subset \P^3$ passing through $L$. The fibre over $t$ is the intersection $\lambda \cap X$ in $\P^3$, which is given by $d$ points counted with multiplicity. Moreover the intersection of a general $\lambda$ with $X$ is reduced, i.e. is given by exactly $d$ points by Bertini's Theorem.
\end{proof}

By the previous lemma, given a point $L \in \P^3 \smallsetminus X$, it makes sense to study the monodromy group of the projection $\pi_L|_X$. For the sake of simplicity, we will not distinguish between $\pi_L$ and its restriction $\pi_{L}|_X$, and both will
be denoted simply by $\pi_L$. We are interested in studying the following spaces:
\begin{align*}
\Wu(X)&:=\{L \in \P^3 \smallsetminus X\ |\ M(\pi_L)=S_d\}\\
\Wn(X)&:=\{L \in \P^3 \smallsetminus X\ |\ M(\pi_L) \neq S_d\}.
\end{align*}
As pointed out in the introduction, we will call the elements of $\Wu(X)$ \textbf{uniform}, and the elements of $\Wn(X)$ \textbf{not uniform}.

\begin{lemma} \label{lemma:lemmainters}
Let $X$ be a surface in $\P^3$, $L \in \Wu(X)$ and $\eta$ be a hyperplane in $\P^3$ containing $L$ such that $X \cap \eta$ is reduced. Then the projection $\pi_L$ can be restricted to $\eta$ and $\Wu(X \cap \eta) \subset \Wu(X) \cap \eta$, where $\Wu(X \cap \eta)$ is the set of points $L \in \eta \smallsetminus X \cap \eta$ for which the projection $\pi_L|_{X \cap \eta}$ is not uniform. In other words $L$ is a uniform point for $X \cap \eta$ implies $L$ is a uniform point for $X$.
\end{lemma}
\begin{proof}
There exists an open set $A$ of the $\P^{2}$ target of $\pi_L$ and a point $x\in A \cap \eta$ such that the construction of the monodromy group can be described by the following commutative diagram
\begin{equation*}
	\xymatrix{
\pi_1(A,x)\ar[r]^m & M(\pi_L)\ar[r] & S_d\ar@{=}[d]\\
\pi_1(A \cap \eta, x)\ar[r]^{m|_\eta}\ar[u]_{i_*} & M(\pi{_L}|_\eta)\ar[r]\ar[u]_{i_*} & S_d\\
}
\end{equation*}
where $M(\pi_L)$ and $M(\pi{_L}|_\eta)$ are just the images of the monodromy maps. The map $i: A \cap \eta \rightarrow A$ is the inclusion and we are denoting the first two vertical maps by $i_*$. If $M(\pi{_L}|_\eta)$ is isomorphic to $S_d$, then the commutativity of the diagram implies also $M(\pi_L)$ to be isomorphic to $S_d$.
\end{proof}

\subsection{Lines having a specified contact with varieties}
We will recall some notion stated in Section $3$ of \cite{CF}. Let $X$ be a smooth hypersurface in $\P^r$ of degree $d$, and consider a line
$\lambda$ not contained in $X$. We will denote by $\lambda_X$ the $0$-dimensional scheme given by $\lambda \cap X$. Considering the multiplicity of the intersection at every point, we write $\lambda_X=\sum n_i x_i$, where $n_i:=m_{x_i}(\lambda,X)$ denotes the intersection multiplicity of the line $\lambda$ at $x_i \in X$. We call the sequence $(n_1, \ldots, n_s)$ the \textbf{intersection type} of $\lambda$ with $X$. Notice that $\sum n_i =d$.
We will use the following notations:
\begin{itemize}
\item $\lambda$ is called \textbf{simple secant} if all the $n_i$'s are equal to $1$.
\item $\lambda$ is called \textbf{tangent} if there is an $n_i \geq 2$; the space of tangent lines to $X$ is denoted by $\cT_X$. If all the $n_i$'s are equal to $1$ except one that is equal to $2$, $\lambda$ is called \textbf{simple tangent}.
\item $\lambda$ is called \textbf{asymptotic tangent} if there is an $n_i \geq 3$; the space of asymptotic tangent lines to $X$ is denoted by $\cF_X$.
\item $\lambda$ is called \textbf{bitangent} if there are $x_i \neq x_j$ such that $n_i, n_j \geq 2$; the space of bitangent lines to $X$ is denoted by $\cB_X$.
\end{itemize}

The branching weight of a line $\lambda$ is defned as $b(\lambda):=\sum (n_i-1)$. Using this, we have that $\lambda$ belongs to $\cT_X$ precisely when $b(\lambda) \geq 1$. Notice that these notions can be generalized to subvarieties of arbitrary dimension (see, for example, \cite{CSJ}). 



We will need a result about the finiteness of the so-called planar points.

\begin{defn}
Let $X$ be defined as above. For a point $x\in X$, consider the curve $C_x:=T_x(X)\cap X$, where $T_x(X)$ is the plane tangent to $X$ in $x$. The point $x$ is called a \textbf{planar point} of $X$ if it is a point of multiplicity $\geq 3$ in $C_x$.
\end{defn}

In particular, this means that every line tangent to $X$ at $x$ will intersect this point with multiplicity at least $3$. For the case of $\P^3$ we have:

\begin{lemma}[\cite{CF}, Lemma 3.6]\label{lemma:planarpts}
Let $X$ be defined as above. There are only a finite number of planes in $\P^3$ cutting $X$ in a curve containing a planar point.
\end{lemma}

Eventually, we recall the following proposition which compares the branch locus of $\pi_L$ with the intersection type of the lines tangent to $X$.

\begin{prop}[\cite{CF}, Proposition 3.8] \label{prop:cfbranch}
Let $X$ be a surface in $\P^3$ and consider the projection $\pi_L$ from a point $L \notin X$ to a general plane in $\P^3$. Consider a point $y$ in the branch locus $B$ of $\pi_L$. Then, the multiplicity of $B$ at $y$ is the branching weight of the line $\langle L, y \rangle$.
\end{prop}

\subsection{Filling families and focal points}
We refer the reader to Section $5$ of \cite{SernFam} for a general introduction to filling families and to Section $1$ of \cite{CilChiant} for a general treatment of the focal locus closely related to our problem. In the following we will use the notation of \cite{CF}.

Let $\X$ be a flat family of lines in $\P^r$, parametrized by an integral base scheme $S$. We have $\X \subset S \times \P^r$, and then we can consider the two projections $q_1$ and $q_2$ restricted to $\X$ as shown in the following diagram:

\begin{equation*}
	\xymatrix{
&  S \times \P^r\ar[dl]_{q_1}\ar[dr]^{q_2} & \\
S & \X\ar[l]^{f_1}\ar[r]_{f_2} & \P^r
}
\end{equation*}
   
\noindent
where we have the maps $f_1:=q_1|_{\X}:\X \to S$ and $f_2:=q_2|_{\X}:\X \to \P^r$. Let $\sT_Z$ denote the sheaf $\cH om(\Omega^1_Z,\sO_Z)$ and $\cN_Z$ the normal sheaf for any scheme $Z$. Then we get the short exact sequence

\begin{equation*}
0 \rightarrow \sT_{\X} \rightarrow \sT_{S\times \P^r\mid \X}\rightarrow \cN_{\X\mid S\times \P^r}\rightarrow 0
\end{equation*}

\noindent
with map induced by $q_2$:
 
 \begin{equation}\label{phi}
\phi: \cH om(\Omega^1_{S\times \P^r\mid \P^r},\sO_{S\times \P^r}) \rightarrow \cN_{\X\mid S \times \P^r}
\end{equation}

\noindent
called the global characteristic map of the aforementioned family.
The map \eqref{phi} and the map induced by the differential $df_2:\sT_{\X} \to q_2^*(\sT_{\P^r}|_{\X})$ have the same kernel (see \cite{CF}).

\begin{defn} \label{defn:focalscheme}
The kernel $ker(\phi)=ker(df_2)=: \cF$  is called the \textbf{focal sheaf} of the family $\X$. Its support $\cF (\X)$ is called the \textbf{focal scheme} of $\X$.
\end{defn}
Notice that the focal sheaf is a torsion sheaf and the dimension of the focal scheme is strictly less than the dimension of its underlying family $\X$. 

\begin{defn} \label{defn:fillingfamily}
Assume now $\X$ to be a family of $h$-dimensional linear subspaces of $\P^r$ parametrised by $S$. We will call $\X$ a \textbf{filling family} if the following conditions are satisfied:
\begin{enumerate}[label=(\roman*)]
\item $\dim(S)=r-h$.
\item the projection $f_2: \X \to \P^r$ is dominant.
\end{enumerate}
\end{defn}

It follows from the definition that the focal scheme describes the set of ramification points of the map $f_2$, so its image under this map actually defines the branch locus.

We have the following:
\begin{prop}[Proposition 4.1, \cite{CF}] \label{prop:CFdegree}
Let $\X$ be a filling family of h-dimensional linear subspaces of $\P^r$, then  the focal scheme of the fibre over the general point $s\in S$, $\cF (\X_s)$, is defined by a hypersurface of degree $r-h$ in $\X_s\cong \P^h$.
\end{prop}

Thus, for 2-dimensional families of lines in $\P^3$ the fibre over the general point is defined by lines $\X_s\cong \P^1$, hence we can describe the focal locus at each fibre as a surface of degree $2$. Consider the map in \eqref{phi} restricted to the fibre over a general point $s\in S$, then we have
\begin{equation*}
\xymatrix{
\sT_{S,s} \otimes \sO_{\X_s}\ar[r]\ar@{}[d]|*=0[@]{\cong} & \cN_{\X_s\mid \P^3}\ar@{}[d]|*=0[@]{\cong}\\
\sO_{\X_s}^{\oplus 2}&\sO_{\X_s}(1)^{\oplus 2}
}
\end{equation*}

\noindent
so that locally the map can be described by a matrix $A_s$ of rank $2$ with entries of degree $1$. Hence the locus cut out by the focal scheme on the general line $\lambda$ of $\X$ is a scheme of dimension $0$ and degree $3-1=2$ defined by $\{det(A_s)=0\}$. The solution will then give either two distinct points of multiplicity $1$ or one of multiplicity $2$.\\

Recall that a surface is said to be developable if it is defined as the locus of lines tangent to a curve or as a cone (for more details, see \cite{GH_LocalDifferential}). We have the following:

\begin{prop}[Proposition 5.1, \cite{CF}] \label{prop:CFtangentfoci}
Let $X$ be a non-developable surface in $\P^3$ and let $\X$ be a filling family of $\P^3$ such that its general member $\lambda$ is tangent to $X$ at a point $x$. Then $x$ defines the focal locus on $\lambda$ and if the contact order of $\lambda$ with $X$ at $x$ is $2$, then $x$ is a focus of multiplicity $2$.
\end{prop}

Finally, recall the following definition.
\begin{defn}\label{defn:fundamentalpts}
Let $\X$ be a filling family of lines in $\P^3$. A point $x \in \P^3$ is called a \textbf{fundamental point} if there is a $1$-dimensional subfamily of $\X$ passing through $x$.
\end{defn}

\section{The case of surfaces in $\P^3$} \label{sec:nupos}
This section is dedicated to the proof of Theorem \ref{thm:main}. As summarized in the introduction, we will first study the focal scheme of the family of lines with branching weight strictly bigger than $1$, as it is done in \cite{CF}; this will lead to the proof Theorem \ref{thm:main}, carried out at the end of this section.

Consider a smooth irreducible surface $X$ in $\P^3$ of degree greater than $1$, and consider the following family of lines in $\P^3$:
$$\cG_X=\{\lambda \in \G(1,3)\ |\ b(\lambda) > 1\}.$$
Moreover, given a point $L \in \P^3 \smallsetminus X$, we want to consider the subfamily
$$\cG_X(L)=\{\lambda \in \G(1,3)\ |\ b(\lambda) > 1 \text{ and } L \in \lambda\}.$$

\noindent
consisting of those elements of $\cG_X$ passing through the point $L$.
Notice that the following subfamilies of the Grassmannian $\G(1,3)$ are algebraic: the family $\cT_X$ of lines tangent to $X$, the family $\cT_X(L)$ of lines tangent to $X$ and passing through $L$, the family $\cG_X$ and $\cG_X(L)$. In particular, $\cG_X(L)$ is also a subfamily of $\cT_X(L)$.

\begin{lemma} \label{lemma:fillingfamily}
$\cG_X$ is a filling family of lines of $\P^3$.
\end{lemma}
\begin{proof}
Let us first study the dimension of $\cG_X$. We can define two subfamilies $\cF_X$ and $\cB_X$ of $\cG_X$ given by flex tangent lines and bitangent lines, respectively.
\begin{align*}
\cF_X&=\{\lambda \in \G(1,\P^3)\ |\ \exists \text{ } q \in X \text{ with } m_q(\lambda,X) \geq 3\}\\
\cB_X&=\{\lambda \in \G(1,\P^3)\ |\ \exists \text{ } p \neq q \in X \text{ with } m_p(\lambda,X), m_q(\lambda,X) \geq 2\}
\end{align*}

\noindent
where $m_q(\lambda,X)$ denotes the intersection multiplicity of the line $\lambda$ at $q\in X$.

Take a general point $x \in X$ and consider the plane tangent to $X$ at $x$. The curve $C_x$ cut out by this plane has a singularity at $x$. We can assume $x$ to be non planar, since Lemma \ref{lemma:planarpts} ensures the number of planar points on $X$ to be finite. Hence, the general line tangent to $X$ at $x$ has contact order $2$. Via local analysis of the singularity of $C_x$ near $x$, we get a finite number of lines $\lambda_i$ which have contact order strictly greater than $2$, and thus they belong to the space $\cF_X$. This shows that $\cF_X$ has dimension two. 
It is also possible to prove that $\cB_X$ has dimension two (see \cite{HarrisAG}, Example 15.21). Either one of these two arguments proves that the family $\cG_X$ has dimension two, and this is the first condition in Definition \ref{defn:fillingfamily} in order to have a filling family. For the second condition, we have to prove that the map $\cG_X \to \P^3$ is dominant.
Notice that, by Proposition \ref{prop:cfbranch}, this is equivalent to asking if the branch locus of the projection of $X$ from a general point of $\P^3$ is singular. Let $d$ be the degree of $X$. We can assume the point $L \in \P^3(x:y:z:t)$ to be $(0:0:0:1)$, such that the projection becomes
$$\pi_L:(x:y:z:t)\mapsto(x:y:z)$$
If $f$ is the polynomial which defines $X$, so that $\deg(f)=d$, the ramification divisor of $\pi_L$ is of the form
$$R:=\{f=\frac{\partial f}{\partial z} = 0\}.$$
Notice that $R$ is a curve, this proves in particular that $\cT_X(L)$ has dimension $1$ for $X$ smooth of degree greater than $1$. 
It follows immediately that $\deg(R)=d(d-1)$. Moreover by the adjunction formula we have
$$K_R=(K_X + R)|_R =(\sO_X(d-4) + R)|_R=\sO_R(2d-5),$$
and then $\deg(K_R)=d(d-1)(2d-5)$. So we have
$$g(R)\leq\frac{1}{2}\left(d(d-1)(2d-5)+2\right).$$

Consider now the branch locus of the projection $B:=\pi_L(R)$. Since the general tangent line to $X$ is simply tangent, $B$ has the same degree of $R$. Moreover $g(B)\leq g(R)$. Since $B$ is planar, it is non-singular if and only if 
$$g(B)=\frac{1}{2}(\deg(B)-1)(\deg(B)-2).$$

By working out the explicit formula in terms of $d$, we get that $B$ is always singular for $d\geq 3$. This shows that also the second condition of Definition \ref{defn:fillingfamily} is satisfied, concluding the proof.
\end{proof}

Let $\cP_X \subset \P^3 \smallsetminus X$ be the locus of points having the property that every tangent line to $X$ passing through one of these points is actually bitangent or flex tangent to $X$. $\cP_X$ is defined as
$$\cP_X=\{L \in \P^3 \smallsetminus X \st \cT_X(L)=\cG_X(L)\}.$$
We want now to consider only lines in $\cG_X$ that actually pass through some point in $\cP_X$. To this aim we will introduce the following incidence variety

$$\cH:=\{(x,\lambda) \subset \cG_X \times \cP_X \st x \in \lambda \},$$
with the two projections
\begin{equation*}
	\xymatrix{
& \cH\ar[dl]_{\pi_1}\ar[dr]^{\pi_2}& \\
\cG_X & & \cP_X
}
\end{equation*}

We are interested in the new family of lines $\widetilde{\cG}_X$, defined as the image under the map $\pi_1$ of $\cH$, or equivalently given by
$$\widetilde{\cG}_X=\{\lambda \in \cG_X\ |\ \exists \text{ } L \in \cP_X \text{ with } L \in \lambda\}$$
Notice that $\widetilde{\cG}_X$ can be also written as the union of $\cG_X(L)$ for all $L$ in $\cP_X$.

\begin{prop} \label{prop:overlinefilling}
If the dimension of $\cP_X$ is greater than or equal to one, then $\widetilde{\cG}_X$ is a filling family of lines in $\P^3$.
\end{prop}
\begin{proof}
Define 
\begin{align*}
U&:= \{x \in \P^3  \text{ such that } x \in \lambda \text{ for a certain } \lambda \in \cG_X\} \subset \P^3,\\
\widetilde{U}&:= \{x \in \P^3  \text{ such that } x \in \lambda \text{ for a certain } \lambda \in \widetilde{\cG}_X\} \subset \P^3,\\
V&:=\{(x,\lambda) \subset \P^3 \times \cG_X  \text{ such that } x \in \lambda \} \subset U \times \cG_X,\\
\widetilde{V}&:=\{(x,\lambda) \subset \P^3 \times \widetilde{\cG}_X  \text{ such that } x \in \lambda \} \subset \widetilde{U} \times \widetilde{\cG}_X.\\
\end{align*}
The situation is then summarized in the following diagram

\begin{equation*}
\xymatrix{
  &  \widetilde{V}\ar[dl]_{\widetilde{p}}\ar@{^{(}->}[d]^{i_2}\ar[dr]^{\widetilde{q}} & & &\\
\widetilde{U}\ar@{^{(}->}[d]^{i_1} & V\ar[dl]^p\ar[dr]_q& \widetilde{\cG}_X\ar@{^{(}->}[d]^{i_3} & \cH\ar[dl]^{\pi_1}\ar[dr]_{\pi_2}\ar[l]_{\pi_1} &\\
U & & \cG_X & &  \cP_X
 }
\end{equation*}

The dimension of the general fibre of $\pi_2$ is $1$ and since by hypothesis the dimension of $\cP_X$ is greater than $1$, we get that the dimension of $\cH$ is greater than $2$. The dimension of $\cG_X$ is also $2$, thus the general fibre of $\pi_1$ has dimension $0$. Hence, the preimage of the general line contained in $\cG_X$ is non empty and the map $i_3$ is dominant. As a result, the dimension of $\widetilde{\cG}_X$ is $2$, proving the first condition of Definition \ref{defn:fillingfamily}.

For the second condition, notice that $V$, $U$ and $\cG_X$ are 
non-empty and that the space $\cG_X$ has dimension at least $2$. The fibre of $q$ over a general line $\lambda \in \cG_X$ has dimension $1$, parametrised by the points on the line $\lambda\cong\P^1$. This implies that the dimension of $V$ is 3.
On the other hand, since $\cG_X$ is filling by Lemma \ref{lemma:fillingfamily}, $U$ has dimension $3$ and hence the fibre over a general point $x \in U$ has dimension zero.

Now let us work out the top part of the diagram. As before, $\widetilde{V}$ has dimension $3$, since the general fibre of $\widetilde{q}$ has dimension $1$. The map $i_3$ is dominant, thus a general element $(x,\lambda) \in V$ also belongs to $\widetilde{V}$. But since the diagram commutes, the dimension of $\widetilde{p}^{-1}(x)$ is equal to the dimension of $p^{-1}(x)$, that is 0. Hence we have that the dimension of $\widetilde{U}$ is $3$ and so the map $\widetilde{\cG}_X \to \P^3$ is dominant.
\end{proof}

\begin{prop} \label{propdimzero}
The dimension of $\cP_X$ is zero.
\end{prop}
\begin{proof}
Notice first that by \cite{GH_LocalDifferential}, since $X$ is smooth of degree greater than or equal to 3, it is non-developable. Moreover, the points of $\cP_X$ are fundamental points of the family $\widetilde{\cG}_X$. According to Section $4$ of \cite{CF}, such points also belong to the focal locus of such a family.

Assume by contradiction that $\cP_X$ has dimension at least one. Then Proposition \ref{prop:overlinefilling} ensures that $\widetilde{\cG}_X$ is a filling family. Consider a general line $\lambda$ in $\widetilde{\cG}_X$: by definition of the family, $\lambda$ would pass through a point in $\cP_X$ that is a focal point. Notice that $\lambda$ can either be a bitangent or a flex tangent to $X$. 
In both cases, by Proposition \ref{prop:CFtangentfoci} we would get either two foci of multiplicity $1$ or one focus of multiplicity $2$. As a result, we would have at least 3 foci (with multiplicities), contradicting Proposition \ref{prop:CFdegree}, for which we should have only 2 solutions to the degree $2$ equation and $A_s$ not identically zero, as the dimension of the focal scheme would be strictly less than the dimension of the family $\X$.
\end{proof}

Let now prove the main result, Theorem \ref{thm:main}. The strategy of the proof consist in showing that the locus $\Wn(X)$ of non-uniform points is contained in the algebraic set $\cP_X$. Proposition \ref{propdimzero} guarantees then that $\cP_X$, and hence $\Wn(X)$, is finite.

\begin{proof}[Proof of Theorem \ref{thm:main}]
Consider a smooth irreducible surface $X$ in $\P^3$. If the degree of $X$ is $1$ or $2$, the result holds trivially because the only possible monodromies are the symmetric groups on $1$ and $2$ elements, respectively. 

Assume the degree $d$ to be greater than or equal to $3$. Consider a point $L \notin \cP_X$, this means that there exists a line $\lambda \in \cT_X(L)$ simple tangent to $X$. For the algebraicity of the family, the dimension of $\cG_X(L)$ will be zero, that is, there are only finitely many elements of $\cT_X(L)$ that are more than simple tangent to $X$.

Hence, if we take the general plane $\eta$ passing through $L$, it cuts $X$ in a curve and does not pass through any of the elements of $\cG_X(L)$. The curve $C:=X \cap \eta$ is smooth and irreducible by Bertini's Theorem. Consider the projection $\pi_L$ and its restriction $\pi_L|_\eta$. By construction, all the lines tangent to $X \cap \eta$ passing through $L$ and are simply tangent to $X \cap \eta$, so, by Lemma 4.6 of \cite{Mir} they correspond to a transposition in the monodromy group $M(\pi_L|_\eta)$. Such a group is then generated by transpositions and so by Lemma \ref{lemma:IndTrans} is isomorphic to $S_d$. Eventually, by Lemma \ref{lemma:lemmainters}, also the group $M(\pi_L)$ is the whole symmetric group, and so $L$ is uniform.

We have proved that $L \notin \cP_X$ implies $L \notin \Wn$, then we have $\Wn \subset \cP_X$.

Proposition \ref{propdimzero} concludes the proof showing that the dimension of $\cP_X$ is zero, hence $\Wn$ is composed by only a finite number of points.
\end{proof}

\section{Cubic surfaces} \label{sec:cubicsurfaces}
Theorem \ref{thm:main} holds also for the case of cubic surfaces. Nevertheless, it is interesting to give a different proof of the result by using automorphisms and moduli spaces. This approach will lead us to the proof of Proposition \ref{prop:genCucCub}, which is the analogous result of \cite{C} for the case of cubic surfaces.

\begin{oss}
Let $X$ be a smooth cubic surface in $\P^3$. If $L$ is not uniform, then the ramification locus $R$ of $\pi_L$ is planar.
Indeed, the Riemann-Hurwitz formula for surfaces gives
$$K_X = f^*K_{\P^2} + 2 R$$
The coefficient $2$ comes from the fact that the preimage under $\pi_L$ of a point in the branch locus consists only of a triple point. Using $K_{\P^2} \cong \sO_{\P^2}(-3)$, $K_X \cong (\sO_{\P^3}(3) + \sO_{\P^3}(-4))|_X = \sO_{\P^3}(-1)|_X$, we get
$$\sO_X(-1) = f^*\sO_{\P^2}(-3) + 2 R$$
Since $X$ is smooth, $f^*\sO_{\P^2}(-3)$ is $\sO_X(-3)$, and so $2 R = \sO_X(2)$ which gives the result, as $R=\sO_X(1)$.
\end{oss}

\begin{prop} \label{prop:smoothcubic}
Let $X$ be a smooth cubic surface. Then the locus $\Wn(X)$ is finite.
\end{prop}
\begin{proof}
Take a general hyperplane $\alpha$ in $\P^3$ and consider the curve $C:=X \cap \alpha$. Assume by contradiction that $W_n$ is of dimension $1$. Then, there would be at least one point $L$, not uniform for $C$. By Theorem $2$ of \cite{VM}, this curve varies maximally in the moduli spaces of planar cubics as $\alpha$ varies. This means that for the general $\alpha$ this curve will have $j$-invariant different from zero. According to \cite{C}, Remark 2.12, this is exactly the hypothesis we need in order to apply Proposition 2.9, proving that $C$ is uniform and getting a contradiction.
\end{proof}

We can exploit the description of the automorphisms of cubic surfaces in order to obtain a result which goes in the direction of \cite{C}.
\begin{prop} \label{prop:genCucCub}
Let $X$ be a general cubic surface. Then the locus $\Wn(X)$ is empty.
\end{prop}
\begin{proof}
For a general cubic surface $X$, it is proven in \cite{H} that the automorphism group $\aut(X)$ is the identity. Let $L$ be a point in $\P^3 \setminus X$ and assume by contradiction $L$ to be non-uniform. Since the monodromy group $M(\pi_L)$ is isomorphic to the Galois group, we have that the extension $[k(X):k(\P^2)]$ is Galois. This gives a non-trivial automorphism of $X$ which leads to a contradiction, concluding the proof.
\end{proof}
\begin{oss}
As an alternative proof of the previous proposition, let again $L$ be a point in $\P^3 \setminus X$ and consider the diagram
$$M(\pi_L) \rightarrow S_3 \leftarrow \aut(X|\P^2)$$
where the symmetric group $S_3$ is the automorphism group of the fibre over a non branch point of $\P^2$ and $\aut(X|\P^2)$ is the space of automorphisms of $X$ fixing the plane $\P^2$. Denote by $D$ the image of $\aut(X|\P^2)$ in $S_3$. As proved in Proposition 1.4 of \cite{C}, $D$ is the centralizer of $M(\pi_L)$ inside $S_3$. Assume by contradiction that there exist a point $L$ which is not uniform. That means $M(\pi_L)$ is the alternating group $A_3$. Since the centralizer of $A_3$ in $S_3$ is the whole $A_3$, we have that $\aut(X|\P^2)$ is not trivial. Hence $\aut(X)$ also is not trivial and that is a contradiction.
\end{oss}

\subsection{The Fermat cubic surface}
A meaningful example for this case, is the Fermat cubic surface. Let $X:=\{F=0\}$ be the Fermat cubic surface, zero locus of $F$ in $\P^3(x_0:x_1:x_2:x_3)$ where 
$$F(x_0:x_1:x_2:x_3)=x_0^3+x_1^3+x_2^3+x_3^3$$
$X$ has automorphism group $\Z/3\Z \rtimes S_4$.

Let us check the point $L:=(0:0:0:1)$ is not uniform. In this case, the projection $\pi_L$ is the map
\begin{align*}
\pi_L: \P^3 \setminus L &\rightarrow \{x_3=0\}\\
(x_0:x_1:x_2:x_3) &\mapsto (x_0:x_1:x_2)
\end{align*}

A point $(x_0:x_1:x_2)$ belongs to the branch locus $B$ of $\pi_L$ if and only if it lies on the Fermat cubic curve  $\{x_0^3+x_1^3+x_2^3=0\}$. This shows that the monodromy is generated by $3$-cycles, hence $M(L)$ must be the alternating group $A_3$.

Moreover, we obtain four points that are not uniform by applying the automorphisms of the cubic surface. We will denote them by $L_i:=(x_j=0)$ for $i\neq j$. The ramification divisor $R_i$ of the point $L_i$ is the Fermat cubic $X \cap \Pi_i$ on the hyperplane $\Pi_i:=\{x_i=0\}$.

\begin{oss}
Notice that Lemma \ref{lemma:fillingfamily} can still be applied in this case, showing that the family $\cG_X$ is still filling. The contradiction used in Proposition \ref{propdimzero} concerns the family $\widetilde{\cG}_X$ that this time is no more filling because $\cP_X$ has dimension zero.
\end{oss}

\section*{Acknowledgements}

The first named author is supported by the Department of Mathematics of the University of Warwick.

The second named author is supported by the Department of Mathematics and Natural Sciences of the University of Stavanger in the framework of the grant 230986 of the Research Council of Norway. Part of this work was carried out during his visit to the Max Planck Institute for Mathematics in Bonn.

This project was first proposed to us by Gian Pietro Pirola and Lidia Stoppino during the graduate summer school \textit{Pragmatic 2016} held at the University of Catania, in Italy in June and July 2016. We would like to thank Pietro and Lidia for inviting us to this rich and fascinating problem and for the many insightful discussions throughout the process. We would also like to express our special thanks to Francesco Russo for all the meaningful and useful suggestions and to Davide Veniani for the remarks on a preliminary version of the paper. Lastly, but not least importantly, we would like to thank the organizing committee of the Pragmatic Summer School for realizing this wonderful learning opportunity.

\bibliographystyle{alpha}

\end{document}